\documentclass[reqno]{amsart}
\usepackage{graphicx,amssymb, mathrsfs,amsmath}
\allowdisplaybreaks
\newtheorem{thm}{Theorem}[section]
\newtheorem{cor}[thm]{Corollary}
\newtheorem{lem}[thm]{Lemma}
\newtheorem{prop}[thm]{Proposition}
\theoremstyle{definition}
\newtheorem{defn}[thm]{Definition}
\theoremstyle{remark}

\numberwithin{equation}{section}

\newcommand{\A}{\mathcal{A}}

\begin{document}

\title[]{Some $s$-numbers of an integral operator of Hardy type in Banach function spaces}%

\author{ David Edmunds, Amiran Gogatishvili,  Tengiz Kopaliani and Nino Samashvili}

\address{ David Edmunds\\
University of Sussex\\
Department of Mathematics\\
Pevensey 2, North-South Road\\
Brighton BN1 9QH,
United Kingdom}
\email{davideedmunds@aol.com}

\address{Amiran Gogatishvili \\
Institute of Mathematics of the Academy of Sciences of the Czech Republic \\
\'Zitna 25 \\
115 67 Prague 1, Czech Republic}
 \email{gogatish@math.cas.cz}

\address{Tengiz Kopaliani \\
Faculty of Exact and Natural Sciences\\
I. Javakhishvili Tbilisi State University\\
 University St. 2\\
 0143 Tbilisi, Georgia}
\email{tengiz.kopaliani@tsu.ge}

\address{ Nino Samashvili\\
Faculty of Exact and Natural Sciences\\
I. Javakhishvili Tbilisi State University\\
 University St. 2\\
 0143 Tbilisi, Georgia}
\email{n.samashvili@gmail.com}

\thanks{ The research was in part supported by the grants no.~31/48 and   no.~DI/9/5-100/13 of the Shota Rustaveli National Science Foundation.
 The research of A.Gogatishvili was partially supported by the grant P201/13/14743S of the Grant agency of the Czech Republic and RVO: 67985840}%

\begin{abstract}

Let $s_{n}(T)$ denote the $n$th approximation, isomorphism, Gelfand, Kolmogorov or Bernstein number
of the Hardy-type integral operator $T$ given by
 $$
 Tf(x)=v(x)\int_{a}^{x}u(t)f(t)dt,\,\,\,x\in(a,b)\,\,(-\infty<a<b<+\infty)
 $$
 and mapping a Banach function space  $E$ to itself. 
 We investigate some geometrical properties of $E$  for which
 $$
C_{1}\int_{a}^{b}u(x)v(x)dx \leq\limsup\limits_{n\rightarrow\infty}ns_{n}(T)
 \leq \limsup\limits_{n\rightarrow\infty}ns_{n}(T)\leq  C_{2}\int_{a}^{b}u(x)v(x)dx
$$
under appropriate conditions on $u$ and $v.$  The constants $C_{1},C_{2}>0$ depend only on the space $E.$
\end{abstract}
\maketitle

   \section{Introduction}
	
	The  $s$-numbers such as approximation, isomorphism, Bernstein, Gelfand and Kolmogorov numbers  $s_n(T)$  of a compact linear map $T$ acting between Banach spaces have proved to give a very useful measure of ‘how compact’ the map is.  For a fine survey of these numbers and their interactions with various parts of mathematics we refer to the monumental book \cite{P} by Pietsch.  The wealth of applications of these ideas has naturally led to the detailed study of $s$-numbers of particular maps, prominent among which are  the weighted Hardy-type operators $ T$, 	for which sharp upper and lower estimates of the approximation numbers  in  $L^p(a,b)$ spaces, $(1\le p\le \infty)$ are investigated  in \cite{EEH1} \cite {EEH2}, \cite{EHL1},  \cite{EHL2} and \cite{NS}.  For  various other $s$-numbers see  \cite{EL1} and \cite{EL2} and the recent book \cite{LE}. When $v=u=1$  (i.e. the non-weighted case) the
problem of the estimation of approximation numbers  for the Hardy operator  acting between  variable exponent Lebesgue spaces   $L^{p(\cdot)}(a,b)$   was  considered  in \cite{ELN}: see the   recent books \cite{LE} and  \cite{ELM}. In Banach function spaces, estimates of approximation numbers were considered in \cite{LS}.
	
	Our purpose in this paper is to  study  $s$-numbers for a weighted Hardy-type  operator $T$ acting in a Banach function space $E$.   Under  some geometrical assumptions on $E$, and on the weights $u$,  $v$, we obtain two-sided estimates for its approximation,   isomorphism,  Bernstein, Gelfand and Kolmogorov numbers.  Our methods of proof are similar to those of \cite{LE} and are   based on the  extension of the estimates  of the function $\A$ (see Section 4) to Banach function spaces under certain geometrical assumptions.

	The paper is organized as follows. Section \ref{s2} contains notation, preliminaries  and formulation of the main results, while in Section \ref{s3}
we  present an application to Lebesgue spaces with variable exponent and in Section \ref{s4}  properties of  the function  $\A$ are established. Estimates of $s$-numbers of the operator are 
 given in Section \ref{s5}. Finally, asymptotic estimates  and the proof of the main result  are  given in Section~\ref{s6}.

	\section{Notation, definitions and Preliminaries}	\label{s2}

 Let $L(I)$ be the space of all Lebesgue-measurable real functions on $I=(a,b),$ where $-\infty<a<b<+\infty.$
 A Banach subspace $E$ of $L(I)$ is said to be a Banach function space (BFS)  if:

1) the norm $\|f\|_{E}$ is defined for every measurable function
$f$ and $f\in E$ if and only if $\|f\|_{E}<\infty:\,\|f\|_{E}=0$
if and only if $f=0$ a.e.;

2) $\||f|\|_{E}=\|f\|_{E}$ for all $f\in E$;

3) if $0\leq f\leq g$ a.e., then $\|f\|_{E}\leq\|g\|_{E}$;

4) if $0\leq f_{n}\uparrow f$ a.e.,then $\|f_{n}\|_{E}\uparrow
\|f\|_{E}$;

5) $ L^{\infty}(I)\subset E\subset L^{1}(I).$



Let $J$ be an arbitrary interval of $I.$ By $E(J)$ we denote the "restriction" of the space $E$
to $J;$ $E(J)=\{f\chi_{J}:\,\,f\in E\}$, with the norm $\|f\|_{E(J)}=\|f\chi_{J}\|_{E}.$

 Given a Banach function space $E,$ its associate space $E'$ consists of those $g\in S$ such that $f\cdot g
\in L^{1}$ for every $f\in E$ with  norm $\|g\|_{E'}=\sup \left\{\|f \cdot
g\|_{L^{1}}:\,\,\|f\|_{E}\leq1\,\, \right\}$. $E'$ is a BFS on $I$ and a closed norm fundamental subspace of the conjugate space
$E^{\ast}$.

We say that the space $E$ has absolutely continuous norm (AC-norm) if for all $f\in E,$
$\|f\chi_{X_{n}}\|_{E}\rightarrow0$ for every sequence of measurable  sets
$\{X_{n}\}\subset I$ such that
$\chi_{X_{n}}\rightarrow0$ a.e. Note that the H\"{o}lder inequality
$$
\int_{I}f(x)g(x)dx\leq\|f\|_{E}\|g\|_{E'}
$$
holds for all $f\in E$ and $g\in E'$ and is sharp (for more details we refer to \cite{BS}).

Let $E$ be a Banach space with dual $E^*$;  the value of  $x^*$ at $x \in E$ is denoted by $(x,x^*)_X $
or $(x,x^*)$ . 

We recall that $E$  is said to be strictly convex if whenever $x,y \in E$ are such that  $x 
\not = y$  and $\|x\| = \|y\| = 1$, and $\lambda \in  (0,1)$, then $\|\lambda x+(1−\lambda )y\| < 1.$
This simply means that the unit sphere in $E$ does not contain any line segment.

 By $\Pi$ we denote the family of all  sequences $\mathcal{Q}=\{I_{i}\}$ of disjoint  intervals  in
$I$ such that
$I=\cup_{I_{i}\in\mathcal{Q}}I_{i}.$ We ignore the
difference in notation caused by a null set.

 Everywhere in the sequel by $l_{\mathcal{Q}},\,\,(\mathcal{Q}\in \Pi)$ we denote  a Banach sequence
space (BSS) (indexed by a partition $\mathcal{Q}=\{I_{i}\}$ of $I$),
 meaning that axioms 1)-4) are satisfied with respect
to the counting measure, and let $\{e_{I_{i}}\}$ denote the standard unit vectors
in $l_{\mathcal{Q}}.$

Throughout the paper we denote by $C,C_1, C_2$ various positive constants
independent of appropriate quantities and not necessarily the
same at each occurrence. By $A\approx B$ we mean that 
$0<C_1\leq A/B\leq C_2<\infty$ for some $C_1,C_2$.

\begin{defn} \label{def1.1}
Let $l=\{l_{\mathcal{Q}}\}_{\mathcal{Q}\in\Pi}$ be a family of BSSs. A BFS E is said to satisfy a uniform upper (lower)
$l-$estimate if there exists a constant $C>0$ such that for every $f\in E$ and $\mathcal{Q}\in\Pi$ we have
$$
\|f\|_{E}\leq
C\|\sum_{I_{i}\in\mathcal{Q}}\|f\chi_{I_{i}}\|_{E}\cdot
e_{I_{i}}\|_{l_{\mathcal{Q}}}\,\,
\left(\|\sum_{I_{i}\in\mathcal{Q}}\|f\chi_{I_{i}}\|_{E}\cdot
e_{I_{i}}\|_{l_{\mathcal{Q}}}\leq C\|f\|_{E}\right).
$$
\end{defn}

Definition \ref{def1.1} was introduced  in \cite{KO}. The idea behind it is simply to generalize the following property of the Lebesgue norm:
$$
\|f\|_{L^{p}}^{p}=\sum_{i}\|f\chi_{\Omega_{i}}\|_{L^{p}}^{p}
$$
for a partition of $\mathbb{R}^{n}$ into measurable sets $\Omega_{i}.$  The notions of uniform upper (lower) $l-$estimates, when
$l_{\mathcal{Q}_{1}}=l_{\mathcal{Q}_{2}}$ for all $\mathcal{Q}_{1},\mathcal{Q}_{2}\in\Pi$, were introduced by
Berezhnoi in \cite{BE}.

 Note that if a BFS  $E$ simultaneously satisfies upper and lower $l=\{l_{\mathcal{Q}}\}_{\mathcal{Q}\in\Pi}$ estimates, then there exists a constant $C>0$ such that, for any $f\in E$ and
 $\mathcal{Q}\in\Pi$,
\begin{equation}\label{min}
\frac{1}{C}\|f\|_{E}\leq\left\|\sum_{I_{i}\in\mathcal{Q}}
\frac{\|f\chi_{I_{i}}\|_{E}}{\|\chi_{I_{i}}\|_{E}} \cdot
\chi_{I_{i}}\right\|_{E}\leq C \|f\|_{E}.
\end{equation}
Note also that if $E$ simultaneously satisfies upper and lower $l=\{l_{\mathcal{Q}}\}_{\mathcal{Q}\in\Pi}$ estimates then $E'$
simultaneously satisfies upper and lower $l'=\{l'_{\mathcal{Q}}\}_{\mathcal{Q}\in\Pi}$ estimates (see \cite{KO}).

We  investigate properties of the Hardy-type operator  of the form
$$
Tf(x)=T_{a,I,u,v}f(x)=v(x)\int_{a}^{x}f(t)u(t)dt,
$$
where $u$ and $v$ are given real valued nonnegative  functions with $|\{x:\,u(x)=0\}|=|\{x:\,v(x)=0\}|=0$  as  a mapping between BFS (by $|\cdot|$ we denote Lebesgue measure).   This operator appears naturally in the theory of differential  equations and it is important to establish when operators of this kind  have properties such as boundedness, compactness, and to estimate their eigenvalues, or their approximation numbers.  We shall assume that  
\begin{align}\label{1.1}
u\chi_{(a,x)}\in E' \\
\intertext{and}
 \chi_{(x,b)}\in E \label{1.2}
\end{align}
 whenever $a<x<b.$

In \cite{KO} the following was proved.
\begin{thm} \label{t1.1}
Let  $E$ and $F$ be BFSs  with the following property:  there exists a family of BSS $l=\{l_{\mathcal{Q}}\}_{\mathcal{Q}\in\Pi}$ such that $E$ satisfies a uniform lower $l$-estimate and $F$ a uniform upper $l$-estimate. Suppose that \eqref{1.1} and \eqref{1.2} hold.  Then $T$ is a bounded   operator from $E$ into $F$ if and only if
$$
\sup_{a<t<b}A(t)=\sup_{a<t<b}\|v\chi_{(t,b)}\|_{F}\|u\chi_{(a,t)}\|_{E'}<\infty.
$$
\end{thm}

We observe that similar results hold when we replace $v$ and $u$ by $v\chi_{J}$ and $u\chi_{J}$ respectively, where $J$ is any
subinterval of $I.$ Note that in \cite{KO} the verification  of the above conditions is carried only  for $I.$ However, the methods of
proof work equally well for arbitrary intervals $J \subset I.$

\begin{thm} \label{t1.2}
Let $J=(c,d)$\label{boundinterval} be any interval of $I$; let  $E$ and $F$ be BFS for which there exists a family of BSS
$l=\{l_{\mathcal{Q}}\}_{\mathcal{Q}\in\Pi}$ such that $E$ satisfies a uniform lower $l$-estimate and $F$ a uniform upper $l$-estimate.
Then the operator
$$
T_{J}f(x)=v(x)\chi_{J}(x)\int_{a}^{x}u(t)\chi_{J}(x)f(t)dt
$$
 is bounded from $E$ into $F$ if and only if
$$
A_{J}=\sup_{t\in J}A_{J}(t)=\sup_{t\in
J}\|v\chi_{J}\chi_{(t,d)}\|_{F}\|u\chi_{J}\chi_{(c,t)}\|_{E'}<\infty.
$$
Moreover $A_{J}\leq\|T_{J}\|\leq K\cdot A_{J},$ where $K\geq1$ is a constant independent of $J.$
\end{thm}

In \cite{EGP} the authors establish a general criterion for $T$ to be compact from $E$ to $F$ when $T:\,E\rightarrow F$ is bounded.
Indeed the following theorem is valid.
\begin{thm}  \label{t1.3}
Let $T:E\rightarrow F$ be bounded, where $E,\,F$ are BFS with  AC- norms.
Then $T$ is compact from $E$ to $E$ if and only if the following two statements are satisfied:
$$
\lim\limits_{x\rightarrow a+}\sup\limits_{a<r<x}\|v\chi_{(r,x)}\|_{F}\|u\chi_{(a,r)}\|_{E'}=0,
$$
and
 $$
\lim\limits_{x\rightarrow b-}\sup\limits_{x<r<b}\|v\chi_{(r,b)}\|_{F}\|u\chi_{(x,r)}\|_{F'}=0.
$$
\end{thm}

Note that if $E$ and $F$ have  AC-norms  and $u\in E',\,v\in F$ then $T:\,E\rightarrow F$ is compact.

More detailed information about the compactness properties of $T$ is provided by the approximation, isomorphism, Bernstein, Gelfand and Kolmogorov numbers and we next recall the definition of those quantities.

 $B(E,F)$ will denote the space of all bounded linear maps of $E$ to $F.$ Given a closed linear subspace $M$ of $E,$  the embedding map of $M$ into $E$ will be denoted by $J_{M}^{E}$ and the canonical map of $E$ onto the quotient space $E/M$ by $Q_{M}^{E}.$ Let $S\in B(E,E).$ Then the  modulus of injectivity of $T$ is
 $$
 j(S)=\sup\{\rho\geq0:\,\,\|Sx\|_{E}\geq\rho\|x\|_{E}\,\,\,\mbox{for all} \quad  x\in E\}.
 $$

\begin{defn}
Let $S \in  B(E,E)$  and $n\in \mathbb{N}.$  Then the $n$th approximation, isomorphism, Gelfand, Bernstein and Kolmogorov numbers
of $S$ are defined by
$$
a_{n}(S)=\inf\{\|S-P\|:\,\,P\in B(E,E),\,\,\mbox{rank}(P)<n\};
$$
$$
i_{n}(S)=\sup\{\|A\|^{-1}\|B\|^{-1}\},
$$
where the supremum is taken over all possible Banach spaces $G$ with $\mbox{dim} G\geq n$ and maps $A\in B(E,G),\,\,B\in
B(G,E)$ such that $ASB$ is the identity on $G;$

$$
c_{n}(S)=\inf\{\|SJ_{M}^{E}\|:\,\,\mbox{codim}(M)<n\};
$$

$$
b_{n}(S)=\sup\{j(SJ_{M}^{E});\,\,\mbox{dim}(M)\geq n\};
$$

$$
d_{n}(S)=\inf\{\|Q_{M}^{E} S\|:\,\,\mbox{dim}(M)<n\}.
$$
respectively.
\end{defn}

Below $s_{n}(S)$ denotes any of the $n$th approximation, isomorphism, Gelfand, Kolmogorov or Bernstein numbers of the operator $S.$
We summaries some of the facts concerning the numbers $s_{n}(S)$  in the following theorem (see \cite{LE}):
\begin{thm} \label{t1.4}
Let $S\in B(E,E)$ and $n\in \mathbb{N}.$ Then
$$
a_{n}(S)\geq c_{n}(S)\geq b_{n}(S)\geq i_{n}(S)
$$
and
$$
a_{n}(S)\geq d_{n}(S)\geq i_{n}(S).
$$
\end{thm}

 The behavior of the $s$-numbers of the Hardy-type operator $T$ is
 reasonably well understood in case $E=F=L^{p}(a,b).$
\begin{thm}  \label{t1.5}
Suppose that $1<p<\infty,\,\,v\in L^{p}(a,b),\,u\in L^{q}(a,b)$ where $1/p+1/q=1.$ Then for $T:\,L^{p}(a,b)\rightarrow L^{p}(a,b)$ we have
$$
\lim\limits_{n\rightarrow\infty}ns_{n}(T)=\frac{1}{2}\gamma_{p}\int_{a}^{b}u(x)v(x)dx,
$$
where $\gamma_{p}=\pi^{-1}p^{1/q}q^{1/p}\sin(\pi/p).$
\end{thm}
 
When $p=2$  and the $s_n$  are  approximation numbers this was first established in \cite{EEH2},  see also \cite{NS}.
The general case, namely that when $1 < p <\infty$, was proved in \cite{EHL2}, where it appears as a special case of results for trees.   
 When $ p = 2$, for “nice”  $u$ and $v$  these results were improved in   \cite{EKL} and more recently extended for
$1 < p <\infty $ in   \cite{L}.

 We say that a space $E$ fulfills the Muckenhoupt condition if for some constant $C>0$ and for all intervals $J\subset I$ we have
 $$
 \frac{1}{|J|}\|\chi_{J}\|_{E}\|\chi_{J}\|_{E'}\leq C.
 $$

 Note that if $E$ fulfills the Muckenhoupt condition, then using  H\"{o}lders inequality  we obtain
 $$
 \frac{1}{|J|}\int_{J}|f(x)|dx\leq C\frac{\|f\chi_{J}\|_{E}}{\|\chi_{J}\|_{E}},
 $$
 and if additionally   $E$ simultaneously satisfies upper and lower
$l=\{l_{\mathcal{Q}}\}_{\mathcal{Q}\in\Pi}$ estimates,  then from \eqref{min}
we obtain
$$
\left\|\sum_{I_{i}\in\mathcal{Q}}
\frac{1}{|I_{i}|}\int_{I_{i}}|f(x)|dx\right\|_{E}\leq C_{1} \|f\|_{E},
$$
where $C_{1}>0$ is an absolute constant independent of the partition $\mathcal{Q}$ of $I.$ If for a space $E$ we have the Muckenhoupt condition and \eqref{min}, we denote this by writing $E\in \mathcal{M}.$  Note that in the case of a reflexive  variable exponent Lebesgue space the condition
$L^{p(\cdot)}\in \mathcal{M}$ implies the boundedness of the Hardy-Littlwood maximal operator in $L^{p(\cdot)}$ (see \cite{CUF}, \cite{DHHR}).

 The main result of this paper is the following theorem.
\begin{thm}\label{main result}
Let $E$ be BFS  belong to the class $\mathcal{M}.$ Let  the spaces $E,\,E^{\ast}$ be strictly convex and  assume that $E$ and $E'$  have AC- norms.  Suppose $u\in E'$, $ v\in E.$ Then there exists constants $C_{1}=C_{1}(E),C_{2}=C_{2}(E)>0$ such that, for the map  $T:E\to E$
$$
C_{1}\int_{a}^{b}u(x)v(x)dx \leq\limsup\limits_{n\rightarrow\infty}ns_{n}(T)
 \leq \limsup\limits_{n\rightarrow\infty}ns_{n}(T)\leq  C_{2}\int_{a}^{b}u(x)v(x)dx.
$$
\end{thm}

\section{Variable exponent Lebesgue spaces} \label{s3}

Given a measurable function
$p(\cdot):(a,b)\rightarrow[1,+\infty),\,\,L^{p(\cdot)}(a,b)$ denotes the set of measurable functions $f$ on $(a,b)$
such that for some $\lambda>0$,
$$
\int_{(a,b)}\left(\frac{|f(x)|}{\lambda}\right)^{p(x)}dx<\infty.
$$
This set becomes a Banach function space when equipped with the
norm
$$
\|f\|_{p(\cdot)}=\inf\left\{\lambda>0:\,\,
\int_{(a,b)}\left(\frac{|f(x)|}{\lambda}\right)^{p(x)}dx\leq1\right\}.
$$

 These spaces and corresponding variable Sobolev spaces $W^{k,p(\cdot)}$ are of interest in their own right, and
 also have applications to partial differential equations and the  calculus of variations.(For more details of results about
 variable exponent Lebesgue spaces we refer to \cite{CUF} and \cite{DHHR}).

 We say that a function $p: (a,b)\to(1,\infty)$ is  log-H\"older continuous if there exists $C> 0$ such that
$$|p(x)-p(y)|\le \frac{C}{\log(e + 1/|x- y|)} \quad\text{for all}\quad  x, y\in (a,b) \quad  \text{and}\quad x\not =y. $$
 Denote by $\mathcal{P}_{\log}$ the set of all log-H\"older continuous exponents that satisfy
$$
p_{-}=\mbox{ess inf}_{x\in(a,b)}p(x)>1,\,\,\,\,  p_{+}=\mbox{ess sup}_{x\in (a,b)}p(x)<\infty.
$$
 Note
that the log-H\"older continuous condition is in fact optimal in the sense of  the modulus of continuity, for boundedness of the Hardy-Littlewood maximal
operator in variable Lebesgue spaces (see \cite{CUF}, \cite{DHHR}).

 We say that a exponent $p(\cdot)\in\mathcal{P}_{\log}$ is strongly log-H\"{o}lder continuous (and write $p(\cdot)\in\mathcal{SP}_{\log}$)
if there is an increasing continuous function  defined on $[0,b-a]$ such that $\lim\limits_{t\rightarrow0+}\psi(t)=0$ and
$$
-|p(x)-p(y)|\ln|x-y|\leq\psi(|x-y|)\,\,\,\,\mbox{for all}\,\,x,y\in(a,b)\,\,\,\mbox{with}\,\,0<|x-y|<1/2.
$$

In \cite{KO} the following was proved.
\begin{prop}\label{minexample}
 Let $ p(\cdot)\in \mathcal{P}_{\log}.$  Then $L^{p(\cdot)}(a,b)\in \mathcal{M}.$
\end{prop}

Note that there  exists another classes  of exponents giving rise to property \eqref{min}.  Indeed,
 let $p(\cdot):[0,1]\rightarrow [1,+\infty)$ be log-H\"older continuous,
  $w(t)=\int_{a}^{t}l(u)du,\,t\in (a,b),
\,\,w(b)=1,\,\, l(u)>0\,\,(u\in (a,b))$. Then $L^{p((w(\cdot))}(a,b)$ has property \eqref{min})  (see \cite{K5}).

  From Theorem  \ref{main result} and Proposition \ref{minexample} we obtain
  \begin{cor} \label{c2.2}
  Let $ p(\cdot)\in\mathcal{P}_{\log}$  and $v\in L^{p(\cdot)}(a,b),\,u\in L^{q(\cdot)}(a,b)\,\,(1/p(x)+1/q(x)=1,\,x\in (a,b)).$ Then $T$ acts from the variable exponent space 
	$ L^{p(\cdot)}(a,b)$  	to itself and 
  $$
 C_{1}\int_{(a,b)}u(x)v(x)dx\leq
 \limsup\limits_{n\rightarrow\infty}ns_{n}(T)
 \leq \limsup\limits_{n\rightarrow\infty}ns_{n}(T)\leq  C_{2}\int_{(a,b)}u(x)v(x)dx.
 $$
  \end{cor}

An analogue of Theorem \ref{main result}  in the setting of spaces with variable exponent when $u=v=1$ was investigated in \cite{ELN}, where the following theorem was proved.
\begin{thm}
Let $p(\cdot)\in\mathcal{SP}_{\log}$ and $u=v=1.$ Then $T$ acts from the variable exponent space 
	$ L^{p(\cdot)}(a,b)$  	to itself and 
$$
\lim\limits_{n\rightarrow\infty}ns'_{n}(T)=\frac{1}{2\pi}\int_{I}(q(x)p(x)^{p(x)-1})^{1/p(x0}\sin(\pi/p(x))dx,
$$
where $s'_{n}(T)$ stands for any of the $n$-th approximation, Gelfand, Kolmogorov and Bernstein numbers of $T.$
\end{thm}

 \section{Properties of $\A$} \label{s4}

Here we establish properties of the function $\A$ which we shall need in the next section.

\begin{defn}\label{d3.1}
Let  $E$  be  a BFS,  $J$ be a subinterval of  $I=(a,b),$  $c\in [a,\,b]$, and suppose that   $u\in E'(J)$ and   $v\in E(J)$.
  We define
$$
\A(J)=\A(J,u,v)=\sup\limits_{f\in E,\,f\neq 0}\inf\limits_{\alpha\in\mathbb{R}}\frac{\|T_{c,J}f-\alpha v\|_{E(J)}}{\|f\|_{E(J)}},
$$
where
$$
T_{c,J}f(x)=v(x)\chi_{J}(x)\int_{c}^{x}f(t)u(t)\chi_{J}(t)dt.
$$
\end{defn}

 We prove some basic properties of  $\A(J).$  Choosing $\alpha=0$  we immediately obtain
$$
\A(J)\leq\|T_{c,J}\|\leq K\cdot A_{J}.
$$
 Note that for $d\in [a,b]$,
$$
T_{d,J}f(x)=T_{c,J}f(x)+v(x)\chi_{J}(x)\int_{d}^{c}f(t)u(t)\chi_{J}(t)dt
$$
and the number $\A(J,u,v)$ is independent of $c\in [a,b].$

\begin{lem}\label{normequal}
Let    $E$ be  a BFS,  $J$  be a subinterval of  $I$, and suppose that  $u\in E'(J)$  and $v\in E(J).$  Set
$$
\widetilde{\A}(J)=\sup\limits_{\|f\|_{E(J)}=1}\inf\limits_{|\alpha|\leq2\|u\|_{E'(J)}}\|T_{c,J}f-\alpha v\|_{E(J)}.
$$
Then $\A(J)=\widetilde{\A}(J).$
\end{lem}
\begin{proof}
 H\"{o}lder's inequality yields
$$
\|T_{c,J}\|\leq\|u\chi_{J}\|_{E'(J)}\|v\chi_{J}\|_{E(J)}.
$$
Let $\|f\|_{E(J)}=1$ and $ |\alpha|>2\|u\|_{E'(J)}.$ Then $|\alpha|>\frac{2\|T_{c,J}\|}{\|v\|_{E(J)}}$ and using the triangle inequality
we obtain
\begin{align*}
\|\alpha v-T_{c,J}f\|_{E(J)}&\geq |\alpha|\|v\|_{E(J)}-\|T_{c,J}\|\|f\|_{E(J)}\\
&>2\|T_{c,J}\|-\|T_{c,J}\|\\
&=\|T_{c,J}\|.
\end{align*}
 We have
\begin{align*}
 \|&T_{c,J}\|\\
&\geq \A(J)\\
&=\sup\limits_{\|f\|_{E(J)}=1}\min\{\inf\limits_{|\alpha|\leq2\|u\|_{E'(J)}}\|T_{c,J}f-\alpha v\|_{E(J)},
 \inf\limits_{|\alpha|>2\|u\|_{E'(J)}}\|T_{c,J}f-\alpha v\|_{E(J)}\}\\
&=\sup\limits_{\|f\|_{E(J)}=1}
 \inf\limits_{|\alpha|\le 2\|u\|_{E'(J)}}\|T_{c,I}f-\alpha v\|_{E(J)}=\widetilde{\A}(J).
\end{align*}
\end{proof}

Note that using the same arguments we may prove that
$$
\A(J)=\sup\limits_{\|f\|_{E(J)}\leq1}\inf\limits_{|\alpha|\leq2\|u\|_{E'(J)}}\|T_{c,J}f-\alpha v\|_{E(J)}.
$$
\begin{lem}
Let    $E$ be a  BFS and  the dual  $E^*$ of the space $E$  has AC- norm.
Let $J$  be a subinterval of  $I$, and suppose that  $u\in E'(J)$ and $v\in E(J).$  Then:

{\rm 1. } The function $\A(x,d)$ is non-increasing and continuous on $(c,d).$

{\rm 2.}   The function $\A(c,x)$ is non-decreasing and continuous on $(c,d).$

{\rm 3.}  $\lim\limits_{x\rightarrow c-}\A(c,x)=\lim\limits_{x\rightarrow d+}\A(x,d)=0.$

\end{lem}

\begin{proof} That $\A(x,d)$ is non-increasing is easy to see.   Fix $y,c<y<d.$ Let $\varepsilon>0.$
Fix $h_{0}>0$ such that  $y-h_{0}>0$ and
$ \|u\|_{E'(y-h,y)}<\varepsilon $ for  $0<h\leq h_{0}.$

Let $D_{h}=\|u\|_{E'(y-h,d)}\,\,\,(0\leq h\leq h_{0})$ and  $w(y)=\chi_{(y,d)}\int_{y-h}^{y}f(t)u(t)dt.$

We have
\begin{align*}
\A(y,d)&\leq\A(y-h,d)\\
&=\sup\limits_{\|f\|_{E(y-h,d)}=1}\inf\limits_{\alpha\in\mathbb{R}}\|\alpha v-T_{y-h,(y-h,d)}f\|_{E((y-h,d))}&\\
&=\sup\limits_{\|f\|_{E(y-h,d)}=1}\inf\limits_{|\alpha|\leq2D_{h}}\{\|(\alpha v-T_{y-h,(y-h,d)}f)\chi_{(y-h,y)}\|_{E((y-h,y))}\\
&\hskip+1cm+\|(\alpha v-T_{y-h,(y-h,d)}f)\chi_{(y,d)}\|_{E((y,d))}\}\\
&\leq\sup\limits_{\|f\|_{E(y-h,d)}=1}\inf\limits_{|\alpha|\leq2D_{h}}\{\|T_{y-h,(y-h,y)}|E((y-h,y))\rightarrow E((y-h,y))\|\times\\
&\hskip+1cm\times \|f\|_{E((y-h,y))}+\|(\alpha v-T_{y,(y-h,d)}f-vw)\chi_{(y,d)}\|_{E((y,d))}\}\\
&\leq\sup\limits_{\|f\|_{E(y-h,d)}=1}\inf\limits_{|\alpha|\leq2D_{h}}\left\{\|u\|_{E'((y-h,y))}\|v\|_{E((y-h,y))}\|f\|_{E((y-h,y))}\right.\\
&\hskip+1cm+\left. \|v\|_{E((y,d))}\|u\|_{E'((y-h,y))}\|f\|_{E((y-h,y))}\right. \\
&\hskip+1cm +\left. \|(\alpha v-T_{y,(y,d)}f)\chi_{(y,d)}\|_{E((y,d))}\right\}\\
&\leq\|v\|_{E((y-h,y))}\varepsilon+\|v\|_{E((y,d))}\varepsilon\\
&\hskip+1cm +\sup\limits_{\|f\|_{E(y-h,d)}=1}\inf\limits_{|\alpha|\leq2D_{h}}\|T_{y,(y,d)}f-\alpha v\|_{E((y,d))}.
\end{align*}
Since $D_{0}\leq D_{h}\leq D_{h_{0}}$
we have
\begin{align*}
\sup\limits_{\|f\|_{E((y-h,d))}=1}&\inf\limits_{|\alpha|\leq2D_{h}}\|T_{y,(y,d)}f-\alpha v\|_{E((y,d))}\\
&\leq\sup\limits_{\|f\|_{E((y-h,d))}=1}\inf\limits_{|\alpha|\leq2D_{0}}\|T_{y,(y,d)}f-\alpha v\|_{E((y,d))}\\
&=\sup\limits_{\|f\|_{E((y,d))}\leq1}\inf\limits_{|\alpha|\leq2D_{0}}\|T_{y,(y,d)}f-\alpha v\|_{E((y,d))}\\
&=\A(y,d)
\end{align*}
and thus
$$
\A(y,d)\leq\A(y-h,d)\leq\|v\|_{E((y-h,y))}\varepsilon+\|v\|_{E((y,d))}\varepsilon+\A(y,d),
$$
which proves that
$$
\lim\limits_{h\rightarrow0+}\A(y-h,d)=\A(y,d).
$$
Analogously
$$
\lim\limits_{h\rightarrow0+}\A(y+h,d)=\A(y,d).
$$

In the same way we prove 2 and 3,
which finishes the proof of the lemma.
\end{proof}

\begin{lem}\label{equivalencenorms}
Let  $E $  be a BFS  satisfying the condition \eqref {min} and suppose that $E'$ has AC-norm. Let $J=(c,d)$  be a subinterval of  $I$,   and suppose that $u\in E'(J)$ and $v\in E(J)$. Then
\begin{equation}\label{3.1}
\A(J)\leq\inf\limits_{x\in J}\|T_{x,J}\,|E(J)\rightarrow E(J)\|.
\end{equation}
 The norms $\|T_{x,J}\|,$ $\|T_{x,(c,x)}\|,$ $\|T_{x,(x,d)}\|$  of the operators $T_{x,J}$ $T_{x,(c,x)},$ $T_{x,(x,d)},$ from $E(J)$ to $E(J),$
are continuous in $x\in (c,d)$ and there exists $e\in J$ such that
\begin{equation}\label{3.2}
\|T_{e,(c,e)}\|=\|T_{e,(e,d)}\|.
\end{equation}
For any $x\in{J}$
\begin{equation}\label{3.3}
\|T_{x,J}\|\approx\max\{\|T_{x,(c,x)}\|,\,\,\|T_{x,(x,d)}\|\},
\end{equation}
and
\begin{equation}\label{3.4}
\min\limits_{x\in J}\|T_{x,J}\|\approx\|T_{e,J}\|.
\end{equation}
\end{lem}
\begin{proof}
For any $x\in (c,d),$
$$
\A(J)\leq\sup\{\|T_{x,J}f\|_{E(J)}:\,\|f\|_{E(J)}=1\}=\|T_{x,J}|\,E(J)\rightarrow E(J)\|,
$$
and consequently we have \eqref{3.1}.

To prove the continuity of $\|T_{x,(x,d)}\|,$  we first note that for $z,y\in(c,d),z<y,$
\begin{align*}
T_{z,(z,d)}f(x)-T_{y,(y,d)}f(x)&=v(x)\chi_{(y,d)}(x)\int_{z}^{y}f(t)u(t)dt\\
&\hskip+0,5cm
+v(x)\chi_{(z,y)}(x)\int_{z}^{y}f(t)u(t)dt.
\end{align*}
Hence, applying H\"{o}lder's inequality,
$$
\|T_{z,(z,d)}-T_{y,(y,d)}\|\leq\|v\|_{E((y,d))}\|u\|_{E'((z,y))}+\|v\|_{E((z,y))}\|u\|_{E'((z,y))}
$$
and so
$$
\left|\|T_{z,(z,d)}\|-\|T_{y,(y,d)}\|\right|\leq\|T_{z,(z,d)}-T_{y,(y,d)}\|
\leq2\|u\|_{E'((z,y)}\|v\|_{E((z,d))},
$$
which yields the continuity of $\|T_{x,(x,d)}\|.$  Similarly we obtain  the continuity
 for $\|T_{x,(c,x)}\|$ and $\|T_{x,J}\|.$

If $\mbox{supp}f\subset (y,d)$ then for $z<y,$
$$
T_{z,(z,d)}f(x)=T_{y,(y,d)}f(x).
$$
Consequently
$$
\|T_{y,(y,d)}\|\leq \|T_{z,(z,d)}\|
$$
and similarly
$$
\|T_{z,(c,z)}\|\leq \|T_{y,(c,y)}\|.
$$

The identity \eqref{3.2} follows from these inequalities and the continuity of the norms $\|T_{x,(c,x)}\|,\|T_{x,(x,d)}\|.$

Let $f\in E(J)$ and set $f_{1}=f\chi_{(c,x)},\,\,f_{2}=f\chi_{(x,d)}.$ Then
$$
(T_{x,J}f)(t)=(T_{x,(c,x)}f_{1})(t)+(T_{x,(x,d)}f_{2})(t).
$$
We have
\begin{align*}
\|T_{x,J}f\|_{E(J)}& \approx \max \{\|T_{x,(c,x)}f_{1}\|_{E((c,x))},\|T_{x,(x,d)}f_{2}\|_{E((x,d))}\}\\
&\leq C \max\{\|T_{x,(c,x)}\|,\|T_{x,(x,d)}\|\}\max \{\|f_{1}\|_{E((c,x))},\|f_{2}\|_{E((x,d))}\}\\
&\leq C\max\{\|T_{x,(c,x)}\|,\|T_{x,(x,d)}\|\}\|f\|_{E(J)}.
\end{align*}
Consequently
$$
\|T_{x,J}\|\leq C\max\{\|T_{x,(c,x)}\|,\|T_{x,(x,d)}\|\}.
$$
The reverse inequality is obvious and \eqref{3.3} is proved.  From \eqref{3.2}, \eqref{3.3} and the above analysis,
we have \eqref{3.4}.
\end{proof}

\begin{defn}
Let  $E $  be a BFS satisfying the condition \eqref{min} and suppose that $E'$ has AC-norm.  Let $J=(c,d)$ be a subinterval of  $I$,  and suppose that $u\in E'(J)$  and  $v\in E(J).$ Define
$$
\widehat{\A}(J)=\|T_{e,(c,e)}\|
$$
where $e\in J$ defined by \ref{3.2}.
\end{defn}

\begin{lem}\label{strictly decreasing and increasing}
Let   $E $ be a BFS satisfying the condition \eqref{min} and suppose that $E'$ has AC-norm; let  $J$  be a subinterval of  $I$,   and suppose that  $u\in E'(J)$ and $v\in E(J)$.
 Then

{\rm 1,}  $\|T_{x,(c,x)}\|$ is strictly increasing  and $\|T_{x,(x,d)}\|$ is strictly decreasing on $(c,d).$

{\rm 2.}   $\widehat{\A}(c,x)$ is strictly increasing  and $\widehat{\A}(x,d)$ is strictly decreasing on $(c,d).$

\end{lem}

\begin{proof}
 The strictly monotonic properties  of the functions  $\|T_{x,(c,x)}\|$ and $\widehat{\A}(c,x)$   follow from the condition  $|\{x:\,u(x)=0\}|=|\{x:\,v(x)=0\||=0.$  If we use arguments analogous to those in the proof  of Lemma \ref{equivalencenorms} we  may prove  continuity of  $\widehat{\A}(c,x).$ 
\end{proof}

\begin{lem}\label{strictly} Let $E$  be a strictly  convex BFS.  Then  given any $f,\,e\in E$, $e\neq 0$
there is a unique scalar $c_{f}$ such that
$$
\|f-c_{f}e\|_{E}=\inf\limits_{c\in\mathbb{R}}\|f-ce\|_{E}.
$$
\end{lem}
\begin{proof}
 Since $\|f-ce\|_{E}$  is continuous in $c$ and tends to $\infty$ as $c\rightarrow\infty$, the
existence of $c_{f}$ is guaranteed by the local compactness of $\mathbb{R}$. The uniqueness of $c_{f}$ follows
from the strict  convexity
of $E.$ 
\end{proof}

\begin{lem}\label{strictlycont}
Let $E$ be a strictly convex BFS and given $f\in E,$ let $c_{f}$ be the unique scalar such that $
\|f-c_{f}e\|_{E}=\inf\limits_{c\in\mathbb{R}}\|f-ce\|_{E},
$ for $e\neq0, \,e\in E.$  Then the map $f\mapsto c_{f}$ is continuous.
\end{lem}
\begin{proof}
Suppose $\|f_{n}-f\|_{E}\rightarrow0.$ Since $c_{f_{n}}$ is bounded, we may suppose that $c_{f_{n}}\rightarrow c.$  Then
$$
\|f_{n}-c_{f}e\|_{E}\geq\|f_{n}-c_{f_{n}}e\|_{E}
$$
and so
$$
\|f-c_{f}e\|_{E}\geq\|f-ce\|_{E}
$$
which gives $c=c_f.$ 
\end{proof}

\begin{lem}\label{strictlyestimate} Let   $E $  be a strictly convex  BFS satisfying  the  condition \eqref{min} and suppose that $E'$ has AC-norm.
 Let $J=(c,d)$  be a subinterval of  $I$,   and suppose that $u\in E'(J)$ and $v\in E(J)$. Then
\begin{equation}\label{3.5}
\A(J)\approx\min\limits_{x\in J}\|T_{x,J}|E(J)\rightarrow E(J)\|\approx\|T_{e,J}|E(J)\rightarrow E(J)\|,
\end{equation}
where $e\in I$ defined by \eqref{3.2}.
\end{lem}
\begin{proof}
  Note that (using \eqref{3.3} and \eqref{3.4})
\begin{align}
\|T_{e,(c,e)}|E(J)\rightarrow E(J)\|&=\|T_{e,(e,d)}|E(J)\rightarrow E(J)\| \notag\\
&\leq\|T_{e,J}|E(J)\rightarrow E(J)\| \notag \\
&\leq C_{1} \|T_{e,(c,e)}|E(J)\rightarrow E(J)\|.\label{3.6}
\end{align}

Let $\alpha<\|T_{e,J}\|.$ Set $T_{e,J}=vF,$ where,
$$
Ff(x)=F_{e,J}f(x)=\chi_{J}(x)\int_{e}^{x}f(t)u(t)\chi_{J}(t)dt.
$$
By \eqref{3.6} it follows that  there exists $f_{i},\,i=1,2,$ supported in $(c,e)$ and $(e,d),$
respectively, such that $\|f_{i}\|_{E}=1,\,\,\|T_{e,J}f_{i}\|_{E(J)}>\alpha/C_{1}$ and $f_{1}$ positive, $f_{2}$ negative.  Note
that the same is true of the signs of the corresponding values of $c_{vFf_{1}},c_{vFf_{2}},$ with $e=v$ (see Lemma \ref{strictly}-\ref{strictlycont}). Hence by the continuity
established in Lemma \ref{strictlycont}, there is a $\lambda\in(0,1)$ such that $c_{vFg}=0$ for $g=\lambda f_{1}+(1-\lambda)f_{2}.$

We have
\begin{align*}
\|T_{e,J}g\|_{E(J)}&\geq C_{2}\max\{\|\lambda T_{e,(c,e)}f_{1}\|_{E((c,e))},\|(1-\lambda) T_{e,(e,d)}f_{2}\|_{E((e,d))}\}\\
&\hskip+1cm\geq C_{3}\alpha\|g\|_{E(J)}.
\end{align*}
We have
$$
\A(J)\geq \inf\limits_{\alpha\in\mathbb{R}}\|vFg-\alpha v\|_{E(J)}/\|g\|_{E(J)}=
\|vFg\|_{E(J)}/\|g\|_{E(J)}\geq C_{3}\alpha.
$$
Since $\alpha< \|T_{e,J}\|$ is arbitrary, $\A(J)\geq C_{3}\|T_{e,J}\|.$   and the first equivalence follows  from  \eqref{3.1}. 
Using \eqref{3.4}, we obtain the second equality of  \eqref{3.5}.
\end{proof}

\begin{lem}\label{approx1}
Let $J=(c,d)$  be a subinterval of  $I$, and suppose that $u_{1}, \, u_{2}$ belong to $E'(J)$ and  $v\in E(J)$. 
Then
$$
|\A(J,u_{1},v)-\A(J,u_{2},v)|\leq \|u_{1}-u_{2}\|_{E'(J)}\|u\|_{E(J)}.
$$
\end{lem}
\begin{proof}
\begin{align*}
\A(J,u_{1},v)&=\sup\limits_{\|f\|_{E(J)}=
1}\inf\limits_{\alpha\in\mathbb{R}}\left\|v(x)\left(\int_{a}^{x}f(t)(u_{1}(t)-u_{2}(t)+u_{2}(t))dt-\alpha\right)\right\|_{E(J)}\\
&\leq\sup\limits_{\|f\|_{E(J)}=1}\inf\limits_{\alpha\in\mathbb{R}}\left[\left\|v(x)\int_{a}^{x}f(t)(u_{1}(t)-u_{2}(t))dt\right\|_{E(J)}\right.\\
&\hskip+1cm\left.
+\left\|v(x)\int_{a}^{x}f(t)u_{2}(t)dt-\alpha v(x)\right\|_{E(J)}\right]\\
&\leq\sup\limits_{\|f\|_{E(J)}=1}\inf\limits_{\alpha\in\mathbb{R}}\left[\|u_{1}-u_{2}\|_{E'(J)}\|u\|_{E(J)}\right.\\
&\hskip+1cm\left. +
\left\|v(x)\int_{a}^{x}f(t)u_{2}(t)dt-\alpha v(x)\right\|_{E(J)}\right]\\
&\leq \|u_{1}-u_{2}\|_{E'(J)}\|u\|_{E(J)}+\A(J,u_{2},v).
\end{align*}
The same holds with $u_{1}$ and $u_{2}$ interchanged, and  the result follows.
\end{proof}

\begin{lem}\label{approx2}
Let $J=(c,d)$  be a subinterval of  $I$, and suppose that   $u\in E'(I)$  and $v_{1}, \, v_2\in E(I)$. 
Then
$$
|\A(J,u,v_{1})-\A(J,u,v_{2})|\leq 3\|v_{1}-v_{2}\|_{E(J)}\|u\|_{E'(J)}.
$$
\end{lem}
\begin{proof}
Let
\begin{align*}
T_{J}^{1}f(x)&=v_{1}(x)\chi_{J}(x)\int_{a}^{x}f(t)u(t)dt,\\
T_{J}^{2}f(x)&=v_{2}(x)\chi_{J}(x)\int_{a}^{x}f(t)u(t)dt,\\
T_{I}^{3}f(x)&=(v_{1}(x)-v_{2}(x))\chi_{J}(x)\int_{a}^{x}f(t)u(t)dt
\end{align*}
Suppose that $\A(J,u,v_{1})>\A(J,u,v_{2}).$ By Lemma \ref{normequal} we have
\begin{align*}
\A(J,u,v_{1})&-\A(J,u,v_{2})\\
&=\sup\limits_{\|f\|_{E(J)}=1}\inf\limits_{\alpha\in\mathbb{R}}\|T_{J}^{1}f-\alpha v_{1}\|_{E(J)}-\A(J,u,v_{2})\\
&=
\sup\limits_{\|f\|_{E(J)}=1}\inf\limits_{|\alpha|2\leq\|u\|_{E(J)}}\|T_{J}^{1}f-\alpha v_{1}\|_{E(J)} -\A(J,u,v_{2})\\
&\leq\sup\limits_{\|f\|_{E(J)}=1}\inf\limits_{|\alpha|\leq2\|u\|_{E(J)}}\left(\|T_{J}^{3}f-\alpha(v_{1}-v_{2})\|_{E(J)}
+\|T_{J}^{2}f-\alpha v_{2}\|_{E(J)}\right)\\
&\hskip+1cm -\A(J,u,v_{2})\\
&\leq\sup\limits_{\|f\|_{E(J)}}\inf\limits_{|\alpha|\leq\|u\|_{E(J)}}
\left(3\|v_{1}-v_{2}\|_{E(J)}\|u\|_{E'(J)}+\|T_{J}^{2}-\alpha v_{2}\|_{E(J)}\right)\\
&\hskip+1cm -\A(J,u,v_{2})\\
&\le 3\|v_{1}-v_{2}\|_{E(J)}\|u\|_{E'(J)}+\A(J,u,v_{2})-\A(J,u,v_{2})\\
&=3\|v_{1}-v_{2}\|_{E(J)}\|u\|_{E'(J)}.
\end{align*}
The proof is complete.
\end{proof}

Note that in Lemma \ref{approx1}-\ref{approx2} we can replace $\A(J)$ by $\|T_{a,J}\|.$

\begin{lem}\label{localestimate}
Let $E\in\mathcal{M}$ be a strictly convex BFS and suppose that $E'$ has AC-  norm. Let $u$ and $v$ be constant over an interval $J=(c,d).$
 Then $\A(J,u,v)\approx uv|J|.$
\end{lem}
\begin{proof}
 From the Muckenhaupt condition we deduce that if $\widetilde{J}\subset J$  and $|\widetilde{J}|/|J|\geq1/2,$ then $\|\chi_{\widetilde{J}}\|_{E}\approx\|\chi_{I}\|_{E}$ and
 $\|\chi_{\widetilde{J}}\|_{E'}\approx\|\chi_{I}\|_{E'}.$
 Let $e\in(c,d),$ we have
$$
\max\left\{\sup\limits_{t\in(c,e)}\|\chi_{(c,t)}\|_{E'}\|\chi_{(t,e)}\|_{E},
\sup\limits_{t\in(e,d)}\|\chi_{(e,t)}\|_{E'}\|\chi_{(t,d)}\|_{E}\right\}\approx|J|.
$$
Using Theorem \ref{boundinterval} and Lemma \ref{strictlyestimate} we obtain
$$
\A(I,1,1)\approx|J|.
$$
Consequently
\begin{align*}
\A(J,u,v)&=\sup\limits_{\|f\|_{E(J)}=1}\inf\limits_{\alpha\in\mathbb{R}}\left\|v\left(\int_{a}^{x}f(t)u)dt-\alpha\right)\right\|_{E(J)}\\
&=uv\sup\limits_{\|f\|_{E(J)}=1}\inf\limits_{c\in\mathbb{R}}\left\|\left(\int_{a}^{x}f(t)dt-c\right)\right\|_{E(J)}\\
&=uv\A(J,1,1)\approx uv|J|.
\end{align*}
\end{proof}

\section{Estimates of $s$-numbers for $T$} \label{s5}

Throughout this section we view $T$  as a map from a BFS $E$ to itself.

\begin{lem}\label{lower estimate}
Let $E $  be a strictly convex  BFS  space that fulfills condition \eqref{min}, let $E'$ have AC-norm,  and suppose that  $u\in E'(I)$  and $v\in E(I)$ .
Let $a=a_{0}<a_{1}<...<a_{N}=b$ be a sequence such that $\mathcal{A}(a_{i-1},a_{i})\leq\varepsilon$ for $i=2,...,N$
and $\|T_{a,(a,a_{1})}\|\leq\varepsilon.$ Then
$$
a_{N}(T)\leq C\varepsilon.
$$
\end{lem}
\begin{proof}
 Set $I_{i}=(a_{i-1},a_{i})$ and $Pf=\sum_{i=2}^{N}P_{i}f,$
where
$$
P_{i}f(x)=v(x)\chi_{I_{i}}\int_{a}^{e_{i}}f(t)u(t)dt,
$$
and $e_{i}$ is a number obtained from Lemma \ref{strictlyestimate} for which
$$
\mathcal{A}(I_{i})=\min\limits_{x\in I_{i}}\|T_{x,I_{i}}|E(I_{i})\rightarrow E(I_{i})\|\approx\|T_{e_{i},I_{i}}|E(I_{i})\rightarrow E(I_{i})\|.
$$

Note that  $\mbox{rank} P\leq N-1$; using Lemma \ref{strictlyestimate} we obtain
\begin{align*}
\|(T-P)f\|_{E}&=\|\chi_{I_{1}}T_{a,I_{1}}f+\sum_{i=2}^{N}(Tf-P_{i}f)\chi_{I_{i}}\|_{E}\\
&=\|\chi_{I_{1}}T_{a,I_{1}}f+\sum_{i=2}^{N}\chi_{I_{i}}T_{e_{i},I_{i}}f\|_{E}\\
&\leq C\|\{\|\chi_{I_{1}}T_{a,I_{1}}f\|_{E},\|\chi_{I_{i}}T_{e_{i},I_{i}}f\|_{E}\}\|_{l}\\
&\leq C \max\{\|T_{a,I_{1}}\|,\,\A(I_{2}),...,\A(I_{N})\}\|\{\|f\chi_{I_{i}}\|_{E}\}\|_{l} \\
&\leq C_{1} \varepsilon\|f\|_{E}.
\end{align*}
\end{proof}

\begin{lem}\label{upper estimate}
Let $E $  be a strictly convex BFS  satisfying condition \eqref{min} . Let    $E^{\ast}$ be strictly convex  and suppose that $E'$ has AC-norm.  Let $u\in E'(I)$ and $v\in E(I)$. 
Let $a=a_{0}<a_{1}<...<a_{N}=b$ be a sequence such that $\mathcal{A}(a_{i-1},a_{i})\geq\varepsilon$ for $i=2,...,N$
and $\|T_{a,(a,a_{1})}\|\geq\varepsilon.$ Then
$$
i_{N}(T)\geq C\varepsilon.
$$
\end{lem}
\begin{proof}
The argument here is similar to the proof of Lemma 6.13 of \cite{LE}
 (which dealt with the case when $ E$ is a Lebesgue space),  but we give full details for the convenience of the reader.
 Set $I_{i}=(a_{i-1},a_{i})\,\,(i=1,...,N).$ From Lemma \ref{strictlyestimate}  it follows that there is
 $e_{i}\in I_{i}$ such that
$$
\mathcal{A}(I_{i})=\min\limits_{x\in I_{i}}\|T_{x,I_{i}}|E(I_{i})\rightarrow E(I_{i})\|\approx\|T_{e_{i},I_{i}}|E(I_{i})\rightarrow E(I_{i})\|.
$$
Note also that (see Lemma \ref{equivalencenorms})
\begin{align*}
\|T_{e_{i},(a_{i-1},e_{i})}|E((a_{i-1},e_{i}))\rightarrow E((a_{i-1},e_{i}))\|&=\|T_{e_{i},(e_{i},a_{i})}|E((e_{i},a_{i}))\rightarrow E((e_{i},a_{i}))\|\\
&\approx\|T_{e_{i},I_{i}}|E(I_{i})\rightarrow E(I_{i})\|.
\end{align*}

Since $T_{e_{i},(a_{i-1},e_{i})}$ and $T_{e_{i},(e_{i},a_{i})}$ are compact operators there exist functions $f^{1}_{i},\,f^{2}_{i}$ such that 
\begin{align*}
\mbox{supp}\,f^{1}_{i}\subset (a_{i-1},e_{i}),\quad \mbox{supp}\,f^{2}_{i}\subset (e_{i},a_{i}),\quad&  \|f^{1}_{i}\|_{E}=\|f^{2}_{i}\|_{E}=1,\\
\|T_{e_{i},(a_{i-1},e_{i})}|E((a_{i-1},e_{i}))\rightarrow E((a_{i-1},e_{i}))\|&=\|T_{e_{i},(a_{i-1},e_{i})}f_{i}^{1}\|_{E(e_{i},(a_{i-1},e_{i}))}\\
\intertext{and}
|T_{e_{i},(e_{i},a_{i})}|E((e_{i},a_{i}))\rightarrow E((e_{i},a_{i}))\|&=|T_{e_{i},(e_{i},a_{i})}f_{i}^{2}\|_{E((e_{i},a_{i}))}.
\end{align*}

 Define $J_{1}=(a_{0},e_{1})=(e_{0},e_{1}),$  $J_{i}=(e_{i-1},e_{i})$ for $i=2,...,N$ and $J_{N+1}=(e_{N-1},b).$
We introduce functions 
\begin{align*}
g_{1}(x)&=f^{1}_{1}(x)\chi_{(e_{0},e_{1})}(x),\\
 g_{i}(x)&=(c_{i}f^{2}_{i-1}(x)\chi_{(e_{i-1},a_{i-1})}(x)+d_{i}f^{1}_{i}(x)\chi_{(a_{i-1},e_{i})}(x))\quad \text{ for } \quad i=2,...,N\\
\intertext{ and }
g_{N+1}(x)&=f^{2}_{N}(x)\chi_{J_{N}}(x).
\end{align*}
For these functions we have
\begin{align*}
\frac{\|T_{e_{i-1},J_{i}}g_{i}\|_{E((e_{i-1},a_{j-1}))}}{\|g_{i}\|_{E((e_{i-1},a_{j-1}))}}&\geq C\varepsilon\\
 \intertext{and}
\frac{\|T_{e_{i},J_{i}}g_{i}\|_{E((a_{i-1},e_{j}))}}{\|g_{i}\|_{E((a_{i-1},e_{j}))}}&\geq C\varepsilon \quad 
\text{for}\quad  i=1,....,N+1.
\end{align*}

We can see that $T_{e_{i-1},J_{i}}g_{i}$ and $T_{e_{i},J_{i}}g_{i}$ do not change sign on $(e_{i-1},a_{i-1})$ and $(a_{i-1},e_{i})$ respectively. Since $T_{e_{i-1},J_{i}}g_{i}(x)$ and $T_{e_{i},J_{i}}g_{i}(x)$ are continuous function we can choose constants $c_{i}$ and $d_{i}$ such that
$$
T_{e_{i-1},J_{i}}g_{i}(a_{i-1})=T_{e_{i},J_{i}}g_{i}(a_{i-1})>0
$$
and $\|g_{j}\|_{E(J_{i})}=1.$ Then we can see that $\mbox{supp}(Tg_{i})\subset J_{i},\,i=2,...,N.$

Note that
\begin{align}
\frac{\|Tg_{i}\|_{E(J_{i})}}{\|g\|_{E(J_{i})}}&=
\frac{\|T_{e_{i-1},(e_{i-1},a_{i-a})}g_{i}\chi_{(e_{i-1},a_{i-a})}+
T_{e_{i},(a_{i-1},e_{i})}g_{i}\chi_{(a_{i-1},e_{i})}\|_{E(J_{i})}}{\|g\|_{E(J_{i})}}  \notag \\
&\approx\frac{\|\{\|T_{e_{i-1},(e_{i-1},a_{i-a})}g_{i}\|_{E((e_{i-1},a_{i-a}))},
\|T_{e_{i},(a_{i-1},e_{i})}g_{i}\|_{E((a_{i-1},e_{i})}\}\|_{l}}{\|g\|_{E(J_{i})}}\notag\\
&\geq C_{1}\varepsilon\quad  \mbox{for}\quad  i=2,...,N.\label{8}
\end{align}

Since $E$ and $E^{\ast}$ are strictly convex BFS,
given any $x\in E\backslash \{0\},$ there is a unique element of $E^{\ast},$ here written as $\widetilde{J}_{E}(x),$
such that $\|\widetilde{J}_{X}(x)\|_{X^{\ast}}=1$  and $<x,\widetilde{J}_{E}(x)>\,=\|x\|_{E}.$ Note that for all
 $x\in E\backslash \{0\},$  $\widetilde{J}_{E}(x)=\mbox{grad}\|x\|_{E},$ where $\mbox{grad}\|x\|_{E}$ denotes the G\^{a}teaux
 derivative of $\|\cdot\|_{E}$ at $x$ (see \cite{LE}).

Denote by $l$ the discrete Banach function space corresponding to the partition $J_{i},\,i=1,...,N+1$ of the interval $I.$
The maps $A:l\rightarrow E$ and $B: E\rightarrow l$ are defined by:
\begin{align*}
A(\{d'_{i}\}_{i=1}^{N})&=\sum_{i=1}^{N+1}d'_{i}g_{i}(x)\\
Bg(x)&=\left\{\frac{<g\chi_{J_{i}},\widetilde{J}_{E}(Tg_{i})>}{\|Tg_{i}\|_{E(J_{i})}}\right\}_{i=1}^{N+1}.
\end{align*}
Since $<Tg_{i},\widetilde{J}_{E}(Tg_{i})>\,=\|Tg_{i}\|_{E},$
$$
BTA(\{d_{i}\}_{i=1}^{N+1})=\{d_{i}\}_{1}^{N+1}.
$$
Observe that $\|B:E\rightarrow l\|$ is attained only for functions of the form
$$
g(x)=\sum_{i=1}^{N+1}c_{i}'Tg_{i}(x),
$$
Using \eqref{8} we obtain
$$
\|g\|_{E}\geq C_{2}\varepsilon\|\{c_{i}'\}_{i=1}^{N+1}\|_{l}
$$
and then
$$
\sup\limits_{\|f\|_{E}\leq1}\|Bf\|_{l}=\sup\limits_{\|g\|_{E}\leq1}\|B(\sum_{i=1}^{N+1}c_{n}'Tg_{i}(x))\|_{l}=
\sup\limits_{\|g\|_{E}\leq1}\|\{c_{i}'\}_{i=1}^{N+1}\|_{l}\leq C_{2}/\varepsilon.
$$
From
$$
\|A(\{d_{i}'\})_{i=1}^{N+1}\|_{E}\approx\|\{\|d_{i}'g_{i}\|_{E(J_{i})}\}\|_{l}=\|\{d_{i}'\}\|_{l}
$$
it follows that $\|A:l\rightarrow E\|\approx 1.$ Thus
$$
i_{N}(T)\geq\|A\|^{-1}\|B\|^{-1}\geq C_{3}\varepsilon.
$$
\end{proof}

Note that in the formulation of Lemmas \ref{lower estimate} and \ref{upper estimate}  instead $\A$ we may use $\widehat{\A}.$

Let  $E $ be a BFS satisfying condition \eqref{min}, let $E'$ have AC-norm, and suppose that $u\in E'(I) $ and $v\in E(I).$
Note that for sufficiently small $\varepsilon>0$ there are $c,d\in (a,b)$ for
which $\widehat{\mathcal{A}}(c,b)=\varepsilon$ and $\|T_{a,(a,d)}\|=\varepsilon.$ Indeed, since $T$ is compact, there
exists a positive integer $N(\varepsilon)$ and points $a=a_{0}<a_{1}<....<a_{N(\varepsilon)}=b$ with $\widehat{\A}(a_{i-1},a_{i})=\varepsilon$ for $i=2,...,N(\varepsilon)-1,\,\,\widehat{\A}(a_{N(\varepsilon)-1},b)\leq\varepsilon$ and
$\|T_{a,(a,a_{1})}\|=\varepsilon.$ The intervals $I_{i}=(a_{i)-1},a_{i}),\,i=1,...,N(\varepsilon)$ form
a partition of $I.$

\begin{lem}\label{large integer values}
Let  $E $  be a BFS satisfying condition \eqref{min}, let $E'$ have AC-norm, and suppose that  $u\in E'(I)$ and $v\in E(I)$. Then the number $N(\varepsilon)$ is a non-increasing function of $\varepsilon$ that takes on every sufficiently large integer value.
\end{lem}
\begin{proof} As in the proof of Lemma 6.11 of \cite{LE},  fix   $c,\,\,a<c<b.$  We have $\|T_{a,(a,c)}\|=\varepsilon_{0}>0$ and there is a positive integer $N(\varepsilon_{0})$ and a partition
$a=a_{0}<a_{1} <...<a_{N(\varepsilon_{0})}=b$ such that $\|T_{a,(a,a_{1})}\|=\varepsilon_{0},$
$\widehat{\A}(a_{i-1},a_{i})=\varepsilon_{0}$ for $i=2,...,N(\varepsilon_{0})-1,\,\,\widehat{\A}(a_{N(\varepsilon_{0})-1},b)\leq\varepsilon_{0}.$ Let $d\in (a,c).$
According to Lemma \ref{strictly decreasing and increasing},  $\widehat{\A}(a,d)=\varepsilon'_{0}<\varepsilon_{0}$ and the procedure
outlined above applied with $\varepsilon'_{0}$ gives $\infty>N(\varepsilon'_{0})\geq N(\varepsilon_{0}).$
By continuity of $\widehat{\A}(c,\cdot)$ and $\|T_{a,(a,\cdot)}\|,$ there exists $d\in(a,c)$ such that $N(\varepsilon'_{0})>N(\varepsilon_{0}).$ If $N(\varepsilon'_{0})=N(\varepsilon_{0})+1,$ stop.
Otherwise, define
$$
\varepsilon_{1}=\sup\{\varepsilon:\,\,\,0<\varepsilon<\varepsilon_{0}\,\,\,\mbox{and}\,\,N(\varepsilon)\geq N(\varepsilon_{0})+1\}.
$$
We claim $N(\varepsilon_{1})=N(\varepsilon_{0})+1.$  Indeed suppose $N(\varepsilon_{1})\geq N(\varepsilon_{0})+2$
and the partition $a=a_{0}<...<a_{N(\varepsilon_{1})}=b$ satisfies $\|T_{a,(a,a_{1})}\|=\varepsilon_{1}$   and $\widehat{\A}(a_{i},a_{i+1})=\varepsilon_{1}$ $i=1,2....,N(\varepsilon_{1})-1$ and $\widehat{\A}(a_{N(\varepsilon_{1})-1},a_{N(\varepsilon_{1})})\leq \varepsilon_{1}.$
Decrease $a_{N(\varepsilon_{1})-1}$ slightly to  $a'_{N(\varepsilon_{1})-1}$ so that $\widehat{\A}(a'_{N(\varepsilon_{1})-1},b)<\varepsilon_{1}$  and
$\widehat{\A}(a_{N(\varepsilon_{1})-2},a'_{N(\varepsilon_{1})-1})>\varepsilon_{1},$ continuing the process to get a partition of
$(a,b)$ having $N(\varepsilon_{1})$ intervals such that $\|T_{a,(a,a_{1})}\|>\varepsilon_{1},\,\,\,\widehat{\A}(a'_{i-1},\,a'_{i})>\varepsilon,\,\,i=2,...,N(\varepsilon_{1})-1$ and
$\widehat{\A}(a_{N(\varepsilon_{1})-1},b)<\varepsilon_{1}.$ Taking $\varepsilon_{2}\leq \min\{\|T_{a,(a,a_{1})}\|,\widehat{\A}(a'_{i-1},\,a'_{i});\,\,i=2,...,N(\varepsilon_{1}-1)\}$ we obtain $\varepsilon_{2}>\varepsilon_{1}$ and $N(\varepsilon_{2})\geq N(\varepsilon_{0})+2$, a contradiction. An inductive argument completes the proof.
\end{proof}

From Lemma \ref{large integer values}, Lemma \ref{strictly decreasing and increasing} and continuity of $\widehat{\A}(c,\cdot)$ and
$\|T_{a,(c,\cdot)}\|$ the next lemma follows.

\begin{lem}\label{there exist varepsilon}
Let  $E $  be a BFS satisfying condition \eqref{min} , let $E'$ have AC-norm, and suppose that  $u\in E'(I)$ and $v\in E(I)$ . Then for each $N>1$ there exist $\varepsilon_{N}$ and a sequence
 $a=a_{0}<a_{1}<....<a_{N}=b$ such that $\widehat{\A}(a_{i-1},a_{i})=\varepsilon_{N}$ for $i=2,...,N$ and $\|T_{a,(a,a_{1})}\|=\varepsilon_{N}.$
\end{lem}

Combining Lemmas \ref{lower estimate}-\ref{there exist varepsilon} we obtain the following theorem.

\begin{thm}\label{number equivalent}
Let   $E $ be a strictly convex  BFS  satisfying condition \eqref{min}, let  $E^{\ast}$  be   strictly convex  and $E'$ have AC-norm, and suppose that  Let $\|u\chi_{I}\|_{E'(I)}\|v\chi_{I}\|_{E(I)}<\infty.$ Then for each $N>1$ there exist $\varepsilon_{N}$ and a sequence $a=a_{0}<a_{1}<....<a_{N}=b$ such that $\A(a_{i-1},a_{i})=\varepsilon_{N}$ for $i=2,...,N$ and $\|T_{a,(a,a_{1})}\|=\varepsilon_{N}$
and
$$
a_{N}(T)\approx i_{N}(T)\approx\varepsilon_{N}.
$$
\end{thm}

\section{Asymptotic results} \label{s6}

\begin{thm}\label{epslon  equivalent}
Let  $E $  be a strictly convex BFS satisfying condition \eqref{min} and suppose it has AC-norm. Let  $E^{\ast}$  be strictly convex, let $E'$  has  AC-norm, and suppose that   $u\in E'(I)$ and $v\in E(I)$. Then there exist constants $C_{1}=C_{1}(E),C_{2}=C_{2}(E)>0$ such that  for the map $T:E\to E$
$$
C_{1}\int_{a}^{b}u(x)v(x)dx \leq\limsup\limits_{n\rightarrow\infty}N\varepsilon_{N}
 \leq \limsup\limits_{n\rightarrow\infty}N\varepsilon_{N}\leq  C_{2}\int_{a}^{b}u(x)v(x)dx
$$
\end{thm}
\begin{proof} As in the proof of Theorem 6.3 of \cite{LE} we observe that for each  $\eta>0$ there exist nonnegative  step functions $u_{\eta},\,v_{\eta}$ on $I$ such that
$$
\|u-u_{\eta}\|_{E'(I)}<\eta,\,\,\,\|u-v_{\eta}\|_{E(I)}<\eta.
$$
We may suppose that
$$
u_{\eta}=\sum_{j=1}^{m}\xi_{j}\chi_{W(j)},\,\,\,v_{\eta}=\sum_{j=1}^{m}\eta_{j}\chi_{W(j)}
$$
where  $W(j)$ are closed subintervals of $I$ with disjoint interiors and $I=\cup_{j=1}^{m}W(j).$

Let $N$ be an integer greater than $1.$ By Lemma \ref{there exist varepsilon} there exist $\varepsilon_{N}>0$ and a sequence
$a_{k},\,k=0,1,...,N,$ such that $a_{0}=a,\,\,a_{N}=b$ and
$$
\widehat{\A}(I_{i})=\varepsilon=\varepsilon_{N}\,\,\,
\mbox{for}\,\,\,i=2,...,N\,\,\mbox{and}\,\,\|T_{a,I_{1}}\|=\varepsilon\,\,\mbox{where}\,\,I_{k}=(a_{k-1},a_{k}).
$$

We have
\begin{align}
\left|\int_{I}u_{\eta}(t)v_{\eta}(t)dt-\int_{I}uv\right |&\leq\int_{I}u(t)|v(t)-v_{\eta}(t)|dt+\int_{I}|u(t)-u_{\eta}(t)|v_{\eta}(t)dt \notag \\
&\leq\|u\|_{E'}\|v-v_{\eta}\|_{E}+\|u-u_{\eta}\|_{E'}\|v\|_{E} \notag \\
&\leq \eta(\|u\|_{E'}+\|v\|_{E}+\eta).\label{5.1}
\end{align}

Let $K=\{k>1:\,\,\mbox{there exists}\,\,j\,\,\mbox{such that}\,\,I_{k}\subset W(j)\}.$  Then $\# K\geq N-1-m,$ and by Lemmas \ref{approx1}-\ref{localestimate},

\begin{align*}
(N-1-m)\varepsilon &\leq C_{1}\sum_{k\in K}\widehat{\A}(I_{k},u,v)\\
&\leq C_{2}\sum_{k\in K}\A(I_{k},u,v)\\
&\leq C_3\sum_{k\in K}\Big\{\A(I_{k},u_{\eta},v_{\eta})\\
&\hskip+1cm+(\A(I_{k},u,v)-\A(I_{k},u_{\eta},v))\\
&\hskip+1cm+ (\A(I_{k},u_{\eta},v)- \A(I_{k},u_{\eta},v_{\eta}))\Big\}\\
&\leq C_{4}\sum_{j}\Big\{|\xi_{j}||\eta_{j}||W(j)|\\
&\hskip+1cm +\|u-u_{\eta}\|_{E'(W(j))}\|v\|_{E(W(j))}\\
&\hskip+1.5cm + \|v-v_{\eta}\|_{E(W(j))}\|u_{\eta}\|_{E'(W(j))}\Big\}\\
&\leq C _{4}\Big(\int_{I}u_{\eta}(t)v_{\eta}(t)dt+\eta\|v\|_{E}+\eta(\|u\|_{E'}+\eta)\Big).
\end{align*}
 
By \eqref{5.1}we conclude that

$$
\limsup\limits_{N\rightarrow\infty}N\varepsilon_{N}\leq C_{4}\Big (\int_{I}u(t)v(t)dt+2\eta\|v\|_{E}+2\eta\|u\|_{E'} +\eta^2\Big)
$$
and then
$$
\limsup\limits_{n\rightarrow\infty}N\varepsilon_{N}\leq C_{4}\int_{I}u(t)v(t)dt.
$$

To prove the opposite inequality we add the end-points of the intervals $W(j),$ $j=1,2,...,m$ to the $a_{k},\,k=0,1,...,N,$ to form
the partition $a=e_{0}<...<e_{n}=b,$ say, where $n\leq N+1+m.$ Note that each interval $J_{i}=(e_{k},e_{k+1})$ is a subinterval of some
$W(j)$ and hence $u_{\eta},v_{\eta}$ have constant values on each $J_{i}.$  Thus
\begin{align*}
\int_{I}u_{\eta}v_{\eta}dt&=\int_{I_{1}}u_{\eta}v_{\eta}dt+\int_{I\backslash I_{1}}u_{\eta}v_{\eta}dt\\
&\le C_{5}\Big(\sum\limits_{J_{i}\subset I_{i}}\|T_{a,J_{i},u_{\eta},v_{\eta}}\|+\sum\limits_{J_{i}\not \subset  I_{i}}\A(J_{i},u_{\eta},v_{\eta})\Big).
\end{align*}

We obtain
\begin{align*}
\sum\limits_{J_{i}\not \subset I_{i}}&\A(J_{i},u_{\eta},v_{\eta}) \\
&\le \sum\limits_{J_{i}\not \subset  I_{i}}\Big\{\A(J_{i},u.v)+(\A(J_{i},u_{\eta}.v)-\A(J_{i},u.v))\\
&\hskip+1cm +
(\A(J_{i},u_{\eta}.v_{\eta})-\A(J_{i},u_{\eta}.v))\Big\}\\
&\leq\sum\limits_{J_{i}\not \subset  I_{i}}\Big \{\A(J_{i},u.v)+\|u-u_{\eta}\|_{E'}\|v\|_{E}+\|u_{\eta}\|_{E'}\|v_{\eta}-v\|_{E}\Big\};
\end{align*}

analogously for $\|T_{a,J,u_\eta},v_{\eta}\|$ we have
\begin{align*}
\sum\limits_{J_{i}\subset I_{i}}&\|T_{a,J_{i},u_{\eta},v_{\eta}}\|\\
&\leq \sum\limits_{J_{i}\subset I_{i}}\Big\{\|T_{a,J_{i},u,v}\|+
(\|T_{a,J_{i},u_{\eta},v}\|-\|T_{a,J_{i},u,v}\|)\\
&\hskip+1cm +(\|T_{a,J_{i},u_{\eta},v_{\eta}}\|-\|T_{a,J_{i},u_{\eta},v}\|)\Big\}\\
&\le \sum\limits_{J_{i}\subset I_{i}}\Big\{\|T_{a,J_{i},u,v}\|+\|u-u_{\eta}\|_{E'}\|v\|_{E}+
\|u_{\eta}\|_{E'}\|v_{\eta}-v\|_{E}\Big\}.
\end{align*}
Hence, from $\|T_{a,J,u,v}\|\leq \varepsilon$ and $\A(J_{i},u,v)\leq C_{5}\varepsilon$
$$
\int_{I}u(t)v(t)dt\leq C_{6}((N+1+m)\varepsilon+3\eta\|v\|_{E}+\eta(3\|u\|_{E'}+\eta))
$$
and since $\eta>0$ is arbitrary the theorem follows.\end{proof}

{\it Proof of Theorem \ref{main result}} 
Combining Theorem \ref{number equivalent} and Theorem \ref{epslon  equivalent}  we obtain  the proof of Theorem \ref{main result}.
\hfill $\square$

\end{document}